\documentclass[preprint,12pt]{elsarticle}

\usepackage{amssymb}
\usepackage{url}
\usepackage{amsthm}
\usepackage{amsmath}

\newtheorem{dfn}{Definition}

\newtheorem{thm}{Theorem}
\newtheorem{prp}{Property}
\newtheorem{prop}{Proposition}
\newtheorem{cor}{Corollary}

\newtheorem{prb}{Problem}

\journal{}

\begin{document}

\begin{frontmatter}

\title{The difference between several metric dimension graph invariants}

\author{Milica Milivojevi\'c Danas\fnref{kg}}
\ead{milica.milivojevic@kg.ac.rs}

\address[kg]{Faculty of Science and Mathematics, University of Kragujevac, Radoja Domanovi\' ca 12, Kragujevac, Serbia }

\begin{abstract}
In this paper extremal values of the difference between several graph invariants related to the metric dimension are studied: mixed metric dimension,
edge metric dimension and strong metric dimension. These non-trivial extremal values are computed over all connected graphs of given order. To obtain such
extremal values  several techniques are developed. They use functions related to metric dimension graph invariants to obtain lower and/or upper bounds on
these extremal values and exact computations when restricting to some specific families of graphs.
\end{abstract}


\end{frontmatter}


\section{Introduction}

If we analyze two graph invariants $\xi_1(G)$ and $\xi_2(G)$, the following question naturally arises:

\begin{prb} \label{pd1} Can the difference between the $\xi_1(G)$ and $\xi_2(G)$ be arbitrarily large?
\end{prb}

This question can be formally presented by introducing function $(\xi_1 - \xi_2)(n)$:

\begin{dfn} \label{dd1} $(\xi_1 - \xi_2)(n)$ is the maximum value of $\xi_1(G) - \xi_2(G)$ over all
graphs $G$ of order $n$.
\end{dfn}

It is easy to see that previous definition is enough, since the mininum value of $\xi_1(G) - \xi_2(G)$ over all
graphs $G$ of order $n$ is equal to $ - (\xi_2 - \xi_1)(n)$.

Describing all results from the literature about extremal difference between graph invariants isn't in this paper scope in full.
It is mentioned only several of differences:
\begin{itemize}
\item Incidence dimension and 2-packing number in graph \cite{dra18};
\item Metric dimension and determining number of a graph \cite{cac10,gar14};
\item Locating-domination number and determining number of a graph \cite{gar14}.
\end{itemize}

\subsection{Metric dimension}

The metric dimension graph invariant was introduced independently by Slater (1975) in \cite{metd1} and Harary and Melter (1976) in \cite{metd2}.
This NP-hard graph invariant (\cite{metdnp}) has been widely investigated in last 55 years and it has applications in many diverse areas.

Given a simple connected undirected graph $G$ with vertex set $V(G)$ and edge set $E(G)$, where $d(u,v)$ denotes the distance between vertices $u$ and $v$, i.e. the length of a shortest $u-v$ path. Then, metric dimension can be defined as:

\begin{dfn} \label{dimd} A vertex $w$ resolves two vertices $u$ and $v$ if $d(u,w) \ne d(v,w)$. A vertex set $S$ of $G$ is a {\em resolving set} of $G$ if every two distinct vertices from $V(G)$ are resolved by some vertex of $S$. Metric basis is resolving set of minimal cardinality, and metric dimension of graph G, denoted as $\beta(G)$, is cardinality of metric basis.
\end{dfn}

Further on, we will present several NP-hard graph invariants based on metric dimension.

\subsection{Strong metric dimension}

The strong metric dimension graph invariant was introduced by Sebo and Tannier \cite{sdim1}:

\begin{dfn} \mbox{\rm(\cite{sdim1})} \label{sdimd} A vertex $w$ strongly resolves two vertices $u$ and $v$ if $u$ belongs to a shortest $v-w$ path or $v$ belongs to a shortest $u-w$ path. A vertex set $S$ of $G$ is a {\em strong resolving set} of $G$ if every two distinct vertices from $V(G) \setminus S$ are strongly resolved by some vertex of $S$.
Strong metric basis is strong resolving set of minimal cardinality, and strong metric dimension of graph G, denoted as $\beta_S(G)$, is cardinality of strong metric basis.
\end{dfn}

The following  property describes link between strong resolving set and resolving set.

\begin{prp} \mbox{\rm(\cite{sdim1})} \label{sgeq}
Each strong resolving set is also a resolving set implying $\beta_S(G)\geq \beta(G)$.
\end{prp}

The following definition of mutually maximally distant vertices and two
properties from the literature will also be used in next section for obtaining lower bound for strong metric dimension.

\begin{dfn} \mbox{\rm(\cite{kra12})} \label{dmmd} A pair of vertices $u,v \in V$, $u \ne v$, is {\em mutually maximally distant}
if and only if
 \begin{itemize}
    \item $d(w,v) \leq d(u,v)$ for each $w \in N(u)$  and
    \item  $d(u,w) \leq d(u,v)$ for each $w \in N(v)$.
 \end{itemize}
\end{dfn}

\begin{prop} \mbox{\rm(\cite{kra12})} \label{smmd} If $S \subset V$ is a strong resolving set of graph $G$, then,
for every two maximally distant vertices $u,v \in V$, it must be $u
\in S$ or $v \in S$.\end{prop}

All additional information about strong metric dimension up to 2014 year can be seen in survey paper \cite{kra14}.

\subsection{Edge and mixed metric dimension}

Distance between edge $uv$ and vertex $w$ is defined as:

\begin{dfn} \mbox{\rm(\cite{edim1})} \label{diste} $d(uv,w)=min\{d(u,w),d(v,w)\}$.
\end{dfn}

The edge metric dimension graph invariant was introduced by Kelenc et al. (2018) in \cite{edim1}:

\begin{dfn} \mbox{\rm(\cite{edim1})} \label{edimd} A vertex $w$ resolves two edges $e_1$ and $e_2$ if $d(e_1,w) \ne d(e_2,w)$. A vertex set $S$ of $G$ is a {\em edge resolving set} of $G$ if every two distinct edges from $E(G)$ are resolved by some vertex of $S$. Edge metric basis is edge resolving set of minimal cardinality, and edge metric dimension
of graph G, denoted as $\beta_E(G)$, is cardinality of edge metric basis.
\end{dfn}

The mixed metric dimension graph invariant was introduced by Kelenc et al. (2017) in \cite{mdim1}:

\begin{dfn} \mbox{\rm(\cite{edim1})} \label{mdimd} A vertex $w$ resolves two items $a$ and $b$ ($a,b \in V(G) \cup E(G)$), if $d(a,w) \ne d(b,w)$. A vertex set $S$ of $G$ is a {\em mixed resolving set} of $G$ if every two distinct items (vertices or edges) from $V(G) \cup E(G)$ are resolved by some vertex of $S$. Mixed metric basis is mixed resolving set of minimal cardinality, and mixed metric dimension of graph G, denoted as $\beta_M(G)$, is cardinality of mixed metric basis.
\end{dfn}

Next property describes link between these graph invariants.

\begin{prp} \mbox{\rm(\cite{mdim1})} \label{geq}
For any graph $G$ it holds $\beta_{M}(G)\geq \max\{\beta(G),\beta_{E}(G)\}$.
\end{prp}

Extremal values of edge and mixed metric dimension are given in next five statements:

\begin{prp} \mbox{\rm(\cite{edim1})} \label{elbub}
For any graph $G$ of order $n$, it holds $1 \le \beta_E(G) \le n-1$.
\end{prp}
\begin{prp} \mbox{\rm(\cite{mdim1})} \label{mlbub}
For any graph $G$ of order $n$, it holds $2 \le \beta_M(G) \le n$.
\end{prp}

\begin{prop} \mbox{\rm(\cite{mdim1})} \label{path}
Let G be any graph of order n. Then $\beta_M(G) = 2$ if and only if G is a path.
\end{prop}

\begin{prop} \mbox{\rm(\cite{edim1})} \label{elb} Let $G$ be any graph and let $\triangle(G)$ be the maximum degree of $G$.
Then $\beta_E(G) \geq \lceil\log_2\triangle (G)\rceil$.
\end{prop}

\begin{prop} \mbox{\rm(\cite{fil19})} \label{elb2} Let $G$ be a connected graph and let $\delta(G)$ be the minimum degree of $G$. Then,
$\beta_E(G) \geq 1 + \lceil log_2 \delta(G) \rceil$.\end{prop}

Mixed metric dimension of trees is completely resolved by Kelenc et al. (2017) in \cite{mdim1}:

\begin{prop} \mbox{\rm(\cite{mdim1})} \label{tree} Let $T$ be any tree with $l(T)$ leaves, then $\beta_M(G) = l(T)$.\end{prop}

It is easy to determine $(\beta_E - \beta_M)(n)$ from facts in the literature.

From Property \ref{geq} it is evident that $\beta_M(G) \geq \beta_E(G)$.
In the literature, it is easy to find graph of order at least 6, where $\beta_M(G) = \beta_E(G)$:

\begin{prp} \mbox{\rm(\cite{edim1})} \label{cbe}
For any complete bipartite graph $K_{r,t}$ different from $K_{1,1}$, it holds $\beta_E(K_{r,t}) = r+t-2$.
\end{prp}

\begin{prp} \mbox{\rm(\cite{mdim1})} \label{cbm}
For any $r,t \ge 3$, it holds $\beta_M(K_{r,t}) = r+t-2$.
\end{prp}

\begin{cor} \mbox{\rm(\cite{mdim1,edim1})} For each $n \ge 6$ it holds $(\beta_E - \beta_M)(n) = 0$.
\end{cor}

Up to now, in the literature there are no any results about $(\beta_M - \beta_E)(n)$ or $(\beta_S - \beta_M)(n)$.
Therefore, new results which give the answer to these questions are given in the next section.

\section{New results}

This section is devoted to presentation of new results about $(\beta_M - \beta_E)(n)$ and $(\beta_S - \beta_M)(n)$.
Answer of general case ($n \ge 4$) for first problem is given in Theorem \ref{men4t}, while answer of special case for $n=3$ is presented in Property \ref{men3t}.
The other problem is resolved in Theorem \ref{smn8t} (general case for $n \ge 7$) and Property \ref{smn3p} (special cases for $3 \le n \le 6$).

\subsection{The difference between mixed and edge metric dimension}

As it is written before, first we resolve special case for $n=3$.

\begin{prp} \label{men3t} $(\beta_M - \beta_E)(3) = 1$\end{prp}
\begin{proof}There are only two connected graphs of order 3: path $P_3$ and cycle $C_3$.
From the facts $\beta_M(P_3) = 2$, $\beta_E(P_3) = 1$, $\beta_M(C_3) = 3$ and $\beta_E(C_3) = 2$,
we have $\beta_M(P_3) - \beta_E(P_3) = \beta_M(C_3) - \beta_E(C_3) = 1$.
Therefore, $(\beta_M - \beta_E)(3) = 1$.
\end{proof}

Next,  general case for $n \ge 4$ is resolved.

\begin{thm} \label{men4t} For each $n \ge 4$ it holds $\lfloor \frac{n}{2} \rfloor - 1 \le (\beta_M - \beta_E)(n) \le n-2$\end{thm}
\begin{proof} First lower bound will be proved. Let $m = \lfloor \frac{n}{2} \rfloor$ and $T'_n$ is a tree given by $V(T'_n) = \{v_1,v_2,...,v_n\}$ and
$E(T'_n) = \{v_iv_{i+1}\,|\,1 \le i \le n-m\} \bigcup \{v_iv_{n-m+i}\,|\,2 \le i \le m\}$. On Figure \ref{men4s}
trees $T'_8$ and $T'_9$ are presented. It is easy to see that tree $T'_n$ has exactly $m+1$ leaves, so by Proposition \ref{tree} from \cite{mdim1}
it directly follows that $\beta_M(T'_n) = l(T'_n) = m+1$.

Next it will be proved, by checking metric coordinates of each edge, that set $S'_E=\{v_1,v_{n-m+1}\}$ is a edge metric basis.
Indeed, $r(v_iv_{i+1},S'_E) = (i-1,n-m-i)$ for $1 \le i \le n-m$ and $r(v_iv_{n-m+i},S'_E) = (i-1,n+1-m-i)$ for $2 \le i \le m$.
It is easy to see that metric coordinates of all edges are mutually different, so $S'_E$ is an edge resolving set.
Since maximal degree of vertices from $T'_n$ is 3 (for example, $N(v_2)=\{v_1,v_3,v_{n-m+2}\}$),
then by Proposition \ref{elb} from \cite{edim1} it holds $\beta_E(T'_n) \geq \lceil\log_2\triangle (T'_n)\rceil = \lceil\log_23\rceil = 2$.
The fact that set $S'_E$ with cardinality 2 is an edge resolving set for $T'_n$ and $\beta_E(T'_n) \geq 2$, directly implying that
$\beta_E(T'_n) = 2$ and set $S'_E$ is a edge metric basis for $T'_n$.

Therefore, $\beta_M(T'_n) = l(T'_n) = m+1$ and $\beta_E(T'_n) = 2$ implying
$(\beta_M - \beta_E)(n) \ge \beta_M(T'_n) - \beta_E(T'_n) = m-1 = \lfloor \frac{n}{2} \rfloor - 1$,
since $(\beta_M - \beta_E)(n)$ is a maximum over all connected graphs of order $n$.

It is evident that connected graphs of order at least 3, have $\triangle (G) \ge 2$. For proving upper bound, we have two cases for graph $G$:\\
{\em Case 1}: $\triangle (G) \ge 3$\\
By Property \ref{mlbub} from \cite{mdim1} every graph $G$ with order $n$ has mixed metric dimension at most $n$.
Also, by Proposition \ref{elb} from \cite{edim1}, $\beta_E(G) \geq \lceil\log_2\triangle (G)\rceil \ge \lceil\log_23\rceil = 2$.
Therefore, $\beta_M(G) - \beta_E(G) \le n-2$.\\
{\em Case 2}: $\triangle (G) = 2$\\
Only connected graph of order $n \ge 4$, with $\triangle (G) = 2$, i.e. degrees of all vertices are less or equal 2, are: path $P_n$ and cycle
$C_n$. Having in mind that $\beta_E(P_n) = 1$, $\beta_M(P_n) = 2$, $\beta_E(C_n) = 2$ and $\beta_M(C_n) = 3$, then in this case holds $\beta_M(G) - \beta_E(G) \le 1 \le n-2$.
Since in both cases holds $\beta_M(G) - \beta_E(G) \le n-2$, we have $(\beta_M - \beta_E)(n) \le n-2$.
\end{proof}

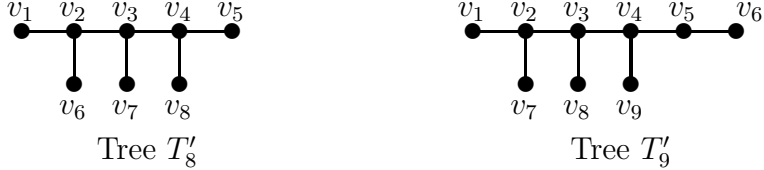
\begin{figure}[htbp]
\begin{center}
\centering\setlength\unitlength{1mm}
\begin{picture}(100,25)
\thicklines
\put(12,0){Tree $T'_8$}
\put(2,17){\circle*{2}}
\put(2,17){\line(1,0){7}}
\put(0,19){$v_1$}
\put(9,17){\circle*{2}}
\put(9,17){\line(1,0){7}}
\put(7,19){$v_2$}
\put(9,10){\circle*{2}}
\put(9,10){\line(0,1){7}}
\put(7,6){$v_6$}
\put(16,17){\circle*{2}}
\put(16,17){\line(1,0){7}}
\put(14,19){$v_3$}
\put(16,10){\circle*{2}}
\put(16,10){\line(0,1){7}}
\put(14,6){$v_7$}
\put(23,17){\circle*{2}}
\put(23,17){\line(1,0){7}}
\put(21,19){$v_4$}
\put(23,10){\circle*{2}}
\put(23,10){\line(0,1){7}}
\put(21,6){$v_8$}
\put(30,17){\circle*{2}}
\put(28,19){$v_5$}

\put(75,0){Tree $T'_9$}
\put(62,17){\circle*{2}}
\put(62,17){\line(1,0){7}}
\put(60,19){$v_1$}
\put(69,17){\circle*{2}}
\put(69,17){\line(1,0){7}}
\put(67,19){$v_2$}
\put(69,10){\circle*{2}}
\put(69,10){\line(0,1){7}}
\put(67,6){$v_7$}
\put(76,17){\circle*{2}}
\put(76,17){\line(1,0){7}}
\put(74,19){$v_3$}
\put(76,10){\circle*{2}}
\put(76,10){\line(0,1){7}}
\put(74,6){$v_8$}
\put(83,17){\circle*{2}}
\put(83,17){\line(1,0){7}}
\put(81,19){$v_4$}
\put(83,10){\circle*{2}}
\put(83,10){\line(0,1){7}}
\put(81,6){$v_9$}
\put(90,17){\circle*{2}}
\put(90,17){\line(1,0){7}}
\put(88,19){$v_5$}
\put(97,17){\circle*{2}}
\put(97,19){$v_6$}

\end{picture}
\caption{Trees $T'_8$ and $T'_9$.}
\label{men4s}
\end{center}
\end{figure}

\subsection{The difference between strong and mixed metric dimension}

As it is written before, first we resolve $(\beta_S - \beta_M)(n)$, for special cases where $3 \le n \le 6$ and that  difference  is presented in the following property.

\begin{prp} \label{smn3p} Exact values, obtained by total enumeration, of $(\beta_M - \beta_E)(n)$, for $3 \le n \le 6$ are given
in Table \ref{smtab}, together with values of $\beta_S(H'_n)$ and $\beta_M(H'_n)$. Extremal graphs $H'_n$ for $3 \le n \le 6$
are presented in Figure \ref{sms}.
\end{prp}

Next, general case for $n \ge 7$ is resolved.

\begin{thm} \label{smn8t} For each $n \ge 7$ it holds $\lfloor \frac{n-1}{2} \rfloor - 2 \le (\beta_S - \beta_M)(n) \le n-4$\end{thm}
\begin{proof} {\bf Step 1.} $(\beta_S - \beta_M)(n) \le n-4$ \\
Let $G$ is connected graph of order $n$ which is not a path. Then from Proposition \ref{path} from \cite{mdim1} it holds
$\beta_M(G) \ge 3$. For each connected graph $G$  is $\beta_S(G) \le n-1$,
therefore $\beta_S(G) - \beta_M(G) \le n-4$.
Other case, when $G=P_n$ is easy: $\beta_S(P_n)=1$ and $\beta_M(P_n)=2$ so $\beta_S(P_n) - \beta_M(P_n) = -1 \le n-4$.
Since in both cases holds $\beta_S(G) - \beta_M(G)  \le n-4$, then $(\beta_S - \beta_M)(n) \le n-4$.

Now, lower bound will be proved. Let $m = \lfloor \frac{n}{2} \rfloor$ and $H'_n$ is a graph given by $V(H'_n) = \{v_1,v_2,...,v_n\}$ and
$E(H'_n) = \{v_iv_{i+1}\,|\,1 \le i \le n-2\} \bigcup \{v_1v_n,v_3v_n,v_{n-1}v_1\}$. On Figure \ref{sms2}
graph $H'_n$ for $n \ge 7$ is graphically presented.

{\bf Step 2.} $\beta_M(H'_n) = 3$ \\
Let $S''_M=\{v_2,v_m,v_{m+3}\}$. For $n=2m$ metric coordinates of vertices are given in (\ref{n2mv}),
while metric coordinates of edges are given in (\ref{n2me}).

\begin{equation} \label{n2mv}
r(v_i,S''_M) = \begin{cases}
(1,m-1,m-3) & i=1 \\
(i-2,m-i,m-4+i) & 2 \le i \le 3 \\
(i-2,m-i,m+3-i) & 4 \le i \le m \\
(m-1,i-m,m+3-i) & m+1 \le i \le m+2 \\
(2m+1-i,i-m,i-m-3) & m+3 \le i \le 2m-1 \\
(2,m-2,m-2) & i=2m \\
\end{cases}
\end{equation}

\begin{equation} \label{n2me}
r(e,S''_M) = \begin{cases}
(0,m-1-i,m-4+i) & e=v_iv_{i+1}, 1 \le i \le 2 \\
(i-2,m-1-i,m+2-i) & e=v_iv_{i+1}, 3 \le i \le m-1 \\
(m-2,0,2) & e=v_mv_{m+1} \\
(2m-i,i-m,m+2-i) & e=v_iv_{i+1}, m+1 \le i \le m+2 \\
(2m-i,i-m,i-m-3) & e=v_iv_{i+1}, m+3 \le i \le 2m-2 \\
(1,m-1,m-4) & e=v_{2m-1}v_1 \\
(1,m-2,m-3) & e=v_{2m}v_1 \\
(1,m-3,m-2) & e=v_{2m}v_3 \\
\end{cases}.
\end{equation}

For $n=2m+1$ metric coordinates of vertices are given in (\ref{n2m1v}),
while metric coordinates of edges are given in (\ref{n2m1e})

\begin{equation} \label{n2m1v}
r(v_i,S''_M) = \begin{cases}
(2-i,m-1,m-3+i) & 1 \le i \le 2 \\
(i-2,m-i,m+3-i) & 3 \le i \le m \\
(i-2,i-m,m+3-i) & m+1 \le i \le m+2 \\
(2m+2-i,i-m,i-m-3) & m+3 \le i \le 2m \\
(2,m-2,m-1) & i=2m+1 \\
\end{cases}
\end{equation}

\begin{equation} \label{n2m1e}
r(e,S''_M) = \begin{cases}
(0,m-2,m-2) & e=v_1v_2 \\
(0,m-3,m-1) & e=v_2v_3 \\
(i-2,m-1-i,m+2-i) & e=v_iv_{i+1}, 3 \le i \le m-1 \\
(m-2,0,2) & e=v_mv_{m+1} \\
(m-1,i-m,m+2-i) & e=v_iv_{i+1}, m+1 \le i \le m+2 \\
(2m+1-i,i-m,i-m-3) & e=v_iv_{i+1}, m+3 \le i \le 2m-1 \\
(1,m-1,m-3) & e=v_{2m}v_1 \\
(1,m-2,m-2) & e=v_{2m+1}v_1 \\
(1,m-3,m-1) & e=v_{2m+1}v_3 \\
\end{cases}.
\end{equation}

From all of the above, since metric representations of all vertices and edges are mutually different, it follows $S''_M$ is a mixed resolving set of graph $H'_n$,
which means that $\beta_M(H'_n) \le 3$. Since $H'_n$ is connected graph, which is not a path, then by Proposition \ref{path} it
holds $\beta_M(H'_n) > 2$. Having in mind that $\beta_M(H'_n)$ has integer value, it implies $\beta_M(H'_n) = 3$.

{\bf Step 3.} $\beta_S(H'_n) \ge \lfloor \frac{n-1}{2} \rfloor + 1$ \\
{\em Case 1}: $n=2m+1$\\
Pair of vertices $v_1,v_{m+1}$ is mutually maximally distant as in Definition \ref{dmmd}, so by Propostion \ref{smmd}
from \cite{kra12}, at least one vertex from that pair must be in any strong resolving set of graph $H'_n$.
The same fact holds for pairs of vertices $v_i,v_{m+i}$ with $3 \le i \le m$.
Therefore, there is $m-1$ disjoint pair of vertices which one member must be in any strong resolving set
of graph $H'_n$.

For remaining 3 vertices $v_2$, $v_{m+2}$ and $v_{2m+1}$, it holds that all 3 pairs $v_2,v_{m+2}$,
$v_2,v_{2m+1}$ and $v_{m+2},v_{2m+1}$ are maximally distant as in Definition \ref{dmmd}. Therefore,
at least two of three vertices $v_2$, $v_{m+2}$ and $v_{2m+1}$ must be in any strong resolving set
of graph $H'_n$. Consequently, any strong resolving set of graph $H'_n$ must have at least $m+1$
vertices, so  $\beta_S(H'_{2m+1}) \ge m+1 = \lfloor \frac{n-1}{2} \rfloor + 1$.

{\em Case 2}: $n=2m$\\
Simiraly as in previous case we have $m-2$ disjoint pairs of mutually maximally distant
vertices as in Definition \ref{dmmd}: $v_i,v_{m+i}$ for $i=1$ and $3 \le i \le m-1$.
Again, for remaining 3 out of 4 vertices $v_2$, $v_{m+2}$ and $v_{2m}$, it holds that all 3 pairs $v_2,v_{m+2}$,
$v_2,v_{2m}$ and $v_{m+2},v_{2m}$ are maximally distant as in Definition \ref{dmmd}.
Consequently, any strong resolving set of graph $H'_n$ must have at least $m$
vertices, so  $\beta_S(H'_{2m+1}) \ge m = \lfloor \frac{n-1}{2} \rfloor + 1$.

Finally, since for $n \ge 7$ it holds $\beta_S(H'_n) \ge \lfloor \frac{n-1}{2} \rfloor + 1$ and
$\beta_M(H'_n) = 3$, then $(\beta_S - \beta_M)(n) \ge \beta_S(H'_n) - \beta_M(H'_n) \ge \lfloor \frac{n-1}{2} \rfloor - 2$. Therefore, the proof on lower bound is completed.
\end{proof}

\begin{table}
\begin{center}
\caption{$(\beta_S - \beta_M)(n)$ for $3 \le n \le 6$}
\label{smtab}
\begin{tabular}{|c|c|c|c|}
\hline
$n$	& $(\beta_S - \beta_M)(n)$ & $\beta_S(H'_n)$ & $\beta_M(H'_n)$\\
\hline
3 & -1 & 1 & 2\\
 \hline
4 & -1 & 2 & 3\\
 \hline
5 & 0 & 3 & 3\\
 \hline
6 & 0 & 3 & 3\\
\hline
\end{tabular}
\end{center}
\end{table}

\begin{figure}[htbp]
\begin{center}
\centering\setlength\unitlength{1mm}
\begin{picture}(100,40)
\thicklines
\put(9,0){$H'_3$}
\put(4,15){\circle*{2}}
\put(4,15){\line(1,0){7}}
\put(2,11){$v_1$}
\put(11,15){\circle*{2}}
\put(11,15){\line(1,0){7}}
\put(9,11){$v_2$}
\put(18,15){\circle*{2}}
\put(16,11){$v_3$}

\put(34,0){$H'_4$}
\put(29,15){\circle*{2}}
\put(27,11){$v_1$}
\put(36,15){\circle*{2}}
\put(34,11){$v_2$}
\put(43,15){\circle*{2}}
\put(41,11){$v_3$}
\put(36,22){\circle*{2}}
\put(36,22){\line(0,-1){7}}
\put(36,22){\line(-1,-1){7}}
\put(36,22){\line(1,-1){7}}
\put(34,24){$v_4$}

\put(59,0){$H'_5$}
\put(54,15){\circle*{2}}
\put(54,15){\line(1,0){7}}
\put(52,11){$v_1$}
\put(61,15){\circle*{2}}
\put(61,15){\line(1,0){7}}
\put(59,11){$v_2$}
\put(68,15){\circle*{2}}
\put(68,15){\line(0,1){7}}
\put(66,11){$v_3$}
\put(68,22){\circle*{2}}
\put(68,22){\line(-1,0){7}}
\put(66,24){$v_4$}
\put(61,22){\circle*{2}}
\put(61,22){\line(-1,-1){7}}
\put(59,24){$v_5$}

\put(84,0){$H'_6$}
\put(79,15){\circle*{2}}
\put(79,15){\line(1,0){7}}
\put(77,11){$v_1$}
\put(86,15){\circle*{2}}
\put(86,15){\line(1,0){7}}
\put(84,11){$v_2$}
\put(93,15){\circle*{2}}
\put(93,15){\line(0,1){7}}
\put(91,11){$v_3$}
\put(93,22){\circle*{2}}
\put(93,22){\line(-1,0){7}}
\put(91,24){$v_4$}
\put(86,22){\circle*{2}}
\put(86,22){\line(-1,-1){7}}
\put(84,24){$v_5$}
\put(79,22){\circle*{2}}
\put(79,22){\line(1,0){7}}
\put(77,24){$v_6$}
\end{picture}
\caption{Extremal graphs $H'_n$ for $2 \le n \le 6$.}
\label{sms}
\end{center}
\end{figure}
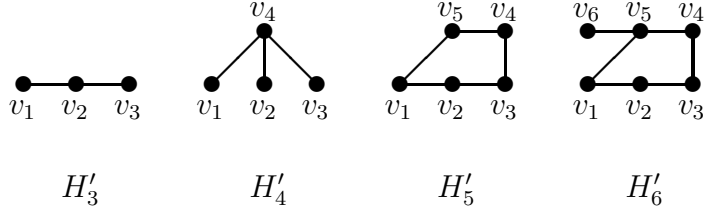

\begin{figure}[htbp]
\begin{center}
\centering\setlength\unitlength{1mm}
\begin{picture}(50,40)
\thicklines
\put(9,0){$H'_n$}
\put(4,22){\circle*{2}}
\put(4,22){\line(1,0){7}}
\put(2,18){$v_1$}
\put(11,22){\circle*{2}}
\put(11,22){\line(1,0){7}}
\put(9,18){$v_2$}
\put(18,22){\circle*{2}}
\put(18,22){\line(1,0){7}}
\put(16,18){$v_3$}
\put(11,15){\circle*{2}}
\put(11,15){\line(-1,1){7}}
\put(11,15){\line(1,1){7}}
\put(9,11){$v_n$}
\put(25,22){\circle*{2}}
\put(25,22){\line(1,0){4}}
\put(31,18){$.\, . \, .$}
\put(23,18){$v_4$}
\put(40,22){\circle*{2}}
\put(40,22){\line(-1,0){4}}
\qbezier(40,22)(22,36)(4,22)
\put(38,18){$v_{n-1}$}
\end{picture}
\caption{Extremal graphs $H'_n$ for $n \ge 7$.}
\label{sms2}
\end{center}
\end{figure}

\section{Conclusions}

This paper is devoted to studying extremal values of the difference between graph invariants related to the metric dimension which were not studied before. First,
the case of difference between mixed and edge metric dimension is resolved. Second, difference between the strong and mixed metric dimension is resolved.
Both questions are answered both for a general case for large $n$, as well as, special cases when $n$ is small.

For a future work, it would be interesting to settle the question about the extremal difference between some metric dimension invariant and its fractional version.
Another direction could be question about the extremal difference between some other graph invariants, which is unsolved up to now.


\end{document}